\documentclass[12pt,final,3p,times]{elsarticle}

\usepackage{amsfonts}
\usepackage{mathrsfs}
\usepackage{amsthm,amsmath,amssymb,mathrsfs,graphicx,graphics,latexsym,exscale,cmmib57,dsfont,bbm,amscd}

\newtheorem*{thmA}{Theorem~A}
\newtheorem*{thmB'}{Theorem~B$^\prime$}
\newtheorem*{thmB''}{Theorem~B$^{\prime\prime}$}
\newtheorem*{thmB'''}{Theorem~B$^{\prime\prime\prime}$}
\newtheorem*{thmB}{Theorem~B}
\newtheorem*{thm3'}{\ref{3}$^\prime$.~Theorem}
\newtheorem*{6C}{\ref{6}C.~Theorem}
\newtheorem*{lem9A}{\ref{9}A.~Lemma}
\newtheorem*{thm9B}{\ref{9}B.~Theorem}
\swapnumbers
\newtheorem{thm}{Theorem}
\newtheorem{lem}[thm]{Lemma}
\newtheorem*{lem*}{Lemma}
\newtheorem*{10A}{\ref{10}A.~Lemma}
\newtheorem*{10C}{\ref{10}C.~Theorem}
\newtheorem{cor}[thm]{Corollary}
\theoremstyle{definition}

\newtheorem{se}[thm]{}

\theoremstyle{definition}

\newtheorem*{6A}{\ref{6}A}
\newtheorem*{6B}{\ref{6}B}
\newtheorem*{10B}{\ref{10}B}

\journal{Topology and its Applications (submitted 8 Feb 2024, accepted 18 Apr 2024)}

\begin{document}

\begin{frontmatter}
\title{Characterizations of open and semi-open maps of compact Hausdorff spaces by induced maps}

\author{Xiongping Dai}
\ead{xpdai@nju.edu.cn}
\author{Yuxuan Xie}
\ead{201501005@smail.nju.edu.cn}
\address{Department of Mathematics, Nanjing University, Nanjing 210093, People's Republic of China}

\begin{abstract}
Let $f\colon X\rightarrow Y$ be a continuous surjection of compact Hausdorff spaces. By
\begin{enumerate}
\item[] $f_*\colon\mathfrak{M}(X)\rightarrow\mathfrak{M}(Y),\ \mu\mapsto \mu\circ f^{-1}$\quad and\quad $2^f\colon2^X\rightarrow2^Y,\ A\mapsto f[A]$
\end{enumerate}
we denote the induced continuous surjections on the probability measure spaces and hyperspaces, respectively. In this paper we mainly show the following facts:
\begin{enumerate}[(1)]
\item If $f_*$ is semi-open, then $f$ is semi-open.
\item If $f$ is semi-open densely open, then $f_*$ is semi-open densely open.
\item $f$ is open iff $2^f$ is open.
\item $f$ is semi-open iff $2^f$ is semi-open.
\item $f$ is irreducible iff $2^f$ is irreducible.
\end{enumerate}
\end{abstract}

\end{frontmatter}

Let $f\colon X\rightarrow Y$ be a continuous map from a topological space $X$ onto another $Y$.
As usual, $f$ is \textit{open} iff the image of every open subset of $X$ is open in $Y$; $f$ is \textit{semi-open}, or \textit{almost open}, iff for every non-empty open subset $U$ of $X$, the interior of $f[U]$, denoted $\mathrm{int}\,f[U]$, is non-empty in $Y$.
The ``open'' and ``semi-open'' properties are important for the structure theory of compact minimal dynamics\,(see, e.g., \cite{V77, Wo, G07, G19}).

By $\mathfrak{M}(X)$, it means the set of all regular Borel probability measures on $X$ equipped with the weak-$*$ topology.
Then, there exists a naturally induced continuous surjective map:
\begin{enumerate}
\item[] $f_*\colon\mathfrak{M}(X)\rightarrow\mathfrak{M}(Y)$,\quad $\mu\mapsto \mu\circ f^{-1}$.
\end{enumerate}
First of all there are equivalent descriptions of openness and semi-openness of $f$ with the help of the induced map of $\mathfrak{M}(X)$ to $\mathfrak{M}(Y)$ as follows:

\begin{thmA}[{Ditor-Eifler~\cite{DE}}]
Let $f\colon X\rightarrow Y$ be a continuous surjection of compact Hausdorff spaces. Then $f$ is open iff the induced map $f_*\colon\mathfrak{M}(X)\rightarrow\mathfrak{M}(Y)$ is an open surjection.
\end{thmA}

\begin{thmB'}[{Glasner~\cite{G07, G19}}]
Let $f\colon X\rightarrow Y$ be a continuous surjection between compact metric spaces. If $f$ is semi-open, then $f_*\colon\mathfrak{M}(X)\rightarrow\mathfrak{M}(Y)$ is a semi-open surjection.
\end{thmB'}
\noindent
In fact, it turns out that condition ``$f$ is semi-open'' is also necessary for that ``$f_*$ is semi-open'' as follows:

\begin{thmB''}
Let $f\colon X\rightarrow Y$ be a continuous surjection of compact Hausdorff spaces. If the induced map $f_*\colon\mathfrak{M}(X)\rightarrow\mathfrak{M}(Y)$ is semi-open, then $f$ is semi-open.
\end{thmB''}

\begin{proof}
Let $U\not=\emptyset$ be an open set in $X$. We shall prove that $\mathrm{int}\,f[U]\not=\emptyset$ in $Y$. For that, let
\begin{enumerate}
\item[] $\mathcal {U}=\{\mu\in\mathfrak{M}(X)\,|\,U\cap\mathrm{supp}(\mu)\not=\emptyset\}$.
\end{enumerate}
Since $\mu\in\mathfrak{M}(X)\mapsto\mathrm{supp}(\mu)\in 2^X$ is lower semi-continuous (cf.~\cite[Lem.~VII.1.4]{Wo}), $\mathcal {U}$ is open in $\mathfrak{M}(X)$. As $f_*$ is semi-open by hypothesis, it follows that there exists an open subset $\mathcal {V}$ of $\mathfrak{M}(Y)$ with $\emptyset\not=\mathcal {V}\subseteq f_*[\mathcal {U}]$.

Because $\overline{\mathrm{co}}(\delta_Y)=\mathfrak{M}(Y)$ where $\delta_Y=\{\delta_y\,|\,y\in Y\}$ is the set of Dirac measures in $Y$, we can choose a measure $\nu=\sum_{i=1}^n\alpha_i\delta_{y_i}\in\mathcal {V}$ with $\alpha_i>0$ for $i=1,\dotsc,n$ and $\sum_i\alpha_i=1$. Further, we can choose an $\varepsilon\in\mathscr{U}_Y$, the uniformity structure of $Y$, such that for all $(y_1^\prime,\dotsc,y_n^\prime)\in\varepsilon[y_1]\times\dotsm\times\varepsilon[y_n]$, there exists an irreducible convex combination
$\nu^\prime=\alpha_1^\prime\delta_{y_1^\prime}+\dotsm+\alpha_n^\prime\delta_{y_n^\prime}\in\mathcal {V}$. By $\mathcal {V}\subseteq f_*[\mathcal {U}]$, it follows that there is some $\mu^\prime\in\mathcal {U}$ with $f_*(\mu^\prime)=\nu^\prime$. As $\mathrm{supp}(\nu^\prime)=\{y_1^\prime,\dotsc,y_n^\prime\}$, $U\cap\mathrm{supp}(\mu^\prime)\not=\emptyset$ and $\mathrm{supp}(\mu^\prime)\subseteq f^{-1}[\mathrm{supp}(\nu^\prime)]\subseteq\bigcup_{i=1}^nf^{-1}(y_i^\prime)$, it follows that
\begin{enumerate}
\item[] $\{y_1^\prime,\dotsc,y_n^\prime\}\cap f[U]\not=\emptyset$.
\end{enumerate}
Clearly, this implies that $\varepsilon[y_i]\subseteq f[U]$ for some $i$ with $1\le i\le n$. Thus, $\mathrm{int}\,f[U]\not=\emptyset$. The proof is completed.
\end{proof}

Consequently we have concluded the following theorem by a combination of Theorem~B$^\prime$ and Theorem~B$^{\prime\prime}$:

\begin{thmB}
Let $f\colon X\rightarrow Y$ be a continuous surjection of compact metric spaces. Then $f$ is semi-open iff $f_*\colon\mathfrak{M}(X)\rightarrow\mathfrak{M}(Y)$ is semi-open.
\end{thmB}

On the other hand, recall that the (largest) {\it hyperspace}, denoted $2^X$, of $X$ is defined to be the collection of all non-empty closed subsets of $X$ equipped with the Vietoris topology\,(see \cite[$\S$II.1]{Wo} or \cite[Thm.~I.1.2]{IN}). Note here that a base for the Vietoris topology is formed by the sets of the form
\begin{enumerate}
\item[] $\langle U_1,\dotsc,U_n\rangle:=\left\{K\in 2^X\,|\,K\subseteq U_1\cup\dotsm\cup U_n, K\cap U_i\not=\emptyset, 1\le i\le n\right\}$
\end{enumerate}
for all $n\ge1$ and all non-empty open sets $U_1$, $\dotsc$, $U_n$ in $X$. Then $X$ is a compact Hausdorff space iff so is $2^X$, and $X$ is metrizable iff $2^X$ is metrizable (cf.~\cite{Mi}, \cite[Thm.~II.1.1]{Wo}, \cite[Thm.~I.3.3]{IN}).

Since Kelley 1942 \cite{K42} the hyperspace theory became an important way of obtaining information on the structure of a topological space $X$ (in continua--compact connected metric spaces) by studying properties of the hyperspace $2^X$ and its hyperspace $2^{2^X}$. In this note we shall give other characterizations of open and semi-open maps with the help of the hyperspaces $2^X$ and $2^Y$ (see Thm.~\ref{3} and Thm.~\ref{4}).
Moreover, we shall consider the interrelation of the irreducibility of $f$ and its induced map $2^f$ (see Thm.~\ref{9}B), and improve Theorem~B$^\prime$ (see Thm.~\ref{10}C).

\begin{se}\label{1}
Let $X$, $Y$ be compact Hausdorff spaces. Let $\phi\colon X\rightarrow Y$ be a continuous surjection. Then $\phi$ induces maps
\begin{enumerate}
\item[] $2^\phi\colon 2^X\rightarrow2^Y$ by $K\in2^X\mapsto\phi[K]\in 2^Y$\quad and\quad $\phi_{\mathrm{ad}}\colon 2^Y\rightarrow2^X$ by $K\in2^Y\mapsto\phi^{-1}[K]\in 2^X$.
\end{enumerate}
Then, $2^\phi$ is a continuous surjection. Moreover, $\phi_{\mathrm{ad}}$ is continuous iff $\phi_{\mathrm{ad}}|_Y=\phi^{-1}$ is continuous where $Y$ is identified with $\{\{y\}\,|\,y\in Y\}\subset 2^Y$ by Lemma below, iff $\phi$ is open; see \cite{Mi} and \cite[Thm.~II.1.3]{Wo}.

\begin{lem*}[{cf.~\cite[Rem.~II.1.4]{Wo}}]
Let $X$ be a compact Hausdorff space and let $n\ge1$ be an integer. Then the map
\begin{enumerate}
\item[] $i_n\colon X^n\rightarrow2^X$ defined by $(x_1,\dotsc,x_n)\mapsto\{x_1,\dotsc,x_n\}$
\end{enumerate}
is continuous. Moreover, it is locally 1-1 in the points $(x_1,\dotsc,x_n)$ with $x_i\not=x_j$ for all $i\not=j$. Also note that $\bigcup\{i_n[X^n]\,|\,n\in\mathbb{N}\}$ is dense in $2^X$.
\end{lem*}

It is natural to wonder about which properties are transmitted between $\phi$ and $2^\phi$. This problem has been addressed by several authors (cf., e.g.,~\cite{Mi, CIM, C}). We shall be concerned with the ``openness'' and ``semi-openness'' here.
\end{se}

\begin{lem}\label{2}
Let $\phi\colon X\rightarrow Y$ be a continuous surjection between Hausdorff spaces with $X$ locally compact. Then:
\begin{enumerate}[(a)]
\item $\phi$ is semi-open iff the preimage of every dense subset of $Y$ is dense in $X$\,(cf.~\cite[Lem.~2.1]{G19}).
\item $\phi$ is semi-open iff the preimage of every residual subset of $Y$ is residual in $X$.
\end{enumerate}
\end{lem}

\begin{proof}
\item (a): If $\phi$ is semi-open and let $A\subset Y$ be a dense set and set $U=X\setminus\overline{\phi^{-1}[A]}$, then $\phi[U]\cap A\not=\emptyset$ whenever $U\not=\emptyset$, a contradiction. Thus, $U=\emptyset$ so that $\phi^{-1}[A]$ is dense in $X$. Conversely, suppose the preimage of every dense subset of $Y$ is dense in $X$. Let $U\not=\emptyset$ be open in $X$. If $\textrm{int}\,\phi[U]=\emptyset$, then $\phi^{-1}[Y\setminus\phi[U]]\cap U\not=\emptyset$, a contradiction. Thus, $\textrm{int}\,\phi[U]\not=\emptyset$.
\item (b)-\textsl{Necessity:} Let $A=\bigcap_{i=1}^\infty A_i$, where $A_i$, $i=1,2,\dots$ are open dense subsets of $Y$. Then, obviously, $\phi^{-1}[A]=\bigcap_i\phi^{-1}[A_i]$. It follows from (a) that $\phi^{-1}[A_i]$ are open dense in $X$. Thus, $\phi^{-1}[A]$ is a residual subset of $X$.
\item (b)-\textsl{Sufficiency:} Let $U\not=\emptyset$ be open in $X$. If $\textrm{int}\,\phi[U]=\emptyset$, then there exists an open set $V$ with $\emptyset\not=V\subseteq\bar{V}\subseteq U$ such that $\bar{V}$ is compact in $X$. Further,
    \begin{enumerate}
    \item[] $\phi^{-1}[Y\setminus\phi[\bar{V}]]\cap \bar{V}\not=\emptyset$\quad and\quad $\phi[\bar{V}]\cap(Y\setminus\phi[\bar{V}])\not=\emptyset$,
    \end{enumerate}
    a contradiction. Thus, $\mathrm{int}\,\phi[U]\not=\emptyset$. The proof is completed.
\end{proof}

\begin{thm}[{cf.~\cite[Thm.~4.3]{H97} for $X,Y$ in continua}]\label{3}
Let $f\colon X\rightarrow Y$ be a continuous surjection between compact Hausdorff spaces. Then $f$ is open iff $2^f\colon 2^X\rightarrow2^Y$ is open.
\end{thm}

\begin{proof}
\item \textsl{Necessity:} Let $\mathcal{A}$ be a member of the base of the hyperspace $2^X$. Then by definition of the Vietoris topology, it follows that there exist nonempty open sets $U_1$, $\dotsc$, $U_n$ in $X$ such that
    \begin{enumerate}
    \item[] $\mathcal{A}=\{K\in 2^X\,|\,K\subseteq U_1\cup\dotsm\cup U_n, K\cap U_i\not=\emptyset\textrm{ for }1\le i\le n\}$.
    \end{enumerate}
    Let $U=U_1\cup\dotsm\cup U_n$. Since $f$ is open, $f[U]$ and $f[U_i]$, $1\le i\le n$, are all open subsets of $Y$. Let
    \begin{enumerate}
    \item[] $\mathcal{B}=\{K\in 2^Y\,|\,K\subseteq f[U], K\cap f[U_i]\not=\emptyset\textrm{ for }1\le i\le n\}$.
    \end{enumerate}
    Clearly, $2^f[\mathcal{A}]\subseteq\mathcal{B}$. In order to prove that $2^f$ is open, it suffices to prove that $2^f[\mathcal{A}]=\mathcal{B}$. For that, we need only prove $\mathcal{B}\subseteq2^f[\mathcal{A}]$.

    Let $B\in\mathcal{B}$ be arbitrarily given. By definitions, there exist points $y_i\in B\cap f[U_i]$ for $i=1,\dotsc,n$. So we can select points $x_i\in U_i$ with $f(x_i)=y_i$ for $1\le i\le n$. Moreover, as $f$ is open and $X, Y$ compact, it follows that there exists a closed set $A^\prime\in 2^X$ with $A^\prime\subseteq U$ such that $f[A^\prime]=B$. Let
    \begin{enumerate}
    \item[] $A=A^\prime\cup\{x_1,\dotsc,x_n\}$.
    \end{enumerate}
    Then $A\in 2^X$ such that $A\subseteq U$ and $x_i\in A\cap U_i\not=\emptyset$ for $i=1,\dotsc,n$. Thus, $A\in\mathcal{A}$ and then $2^f[\mathcal{A}]=\mathcal{B}$.

\item \textsl{Sufficiency:} Let $V\subseteq X$ be an open nonempty set. We need prove that $f[V]$ is open in $Y$. As
\begin{enumerate}
\item[] $\langle V\rangle=\{K\in2^X\,|\,K\subseteq V\}$
\end{enumerate}
is an open subset of $2^X$ and $2^f\colon2^X\rightarrow2^Y$ is open, it follows that $2^f[\langle V\rangle]\subseteq 2^Y$ is open. Since $V=\bigcup\{F\,|\,F\in\langle V\rangle\}$, hence $f[V]=\bigcup\{K\,|\,K\in2^f[\langle V\rangle]\}$. Given $y\in f[V]$, $\{y\}\in 2^f[\langle V\rangle]$ and there exists an open neighborhood $\langle V_{y,1},\dotsc, V_{y,n}\rangle$ of $\{y\}$ in $2^Y$ such that $\{y\}\in\langle V_{y,1},\dotsc, V_{y,n}\rangle\subseteq 2^f[\langle V\rangle]$. Then
\begin{enumerate}
\item[] $f[V]\subseteq\bigcup_{y\in f[V]}(\langle V_{y,1},\dotsc, V_{y,n}\rangle)\subseteq\bigcup\{K\,|\,K\in2^f[\langle V\rangle]\}=f[V]$.
\end{enumerate}
Thus, $f[V]=\bigcup_{y\in f[V]}(\langle V_{y,1},\dotsc, V_{y,n}\rangle)$ is open in $Y$, for each $\langle V_{y,1},\dotsc, V_{y,n}\rangle$ is open in $Y$. The proof is completed.
\end{proof}

Let $f\colon X\rightarrow Y$ be a continuous surjection between compact Hausdorff spaces. Now define
\begin{enumerate}
\item[] $2^{X,f}=\{A\in 2^X\,|\,\exists y\in Y\textrm{ s.t. } A\subseteq f^{-1}(y)\}$,
\end{enumerate}
which is called the \textit{quasifactor representation} of $Y$ in $X$; and, as $i_1\colon Y\rightarrow2^Y$ is an embedding (see Lem.~\ref{1}), we may identify $Y$ with $i_1[Y]=\{\{y\}\,|\,y\in Y\}\subseteq2^Y$ as mentioned before. Thus,
\begin{enumerate}
\item[] $f^\prime=2^f|_{2^{X,f}}\colon 2^{X,f}\rightarrow Y$.
\end{enumerate}
is a well-defined continuous surjection, called the \textit{quasifactor} of $f$. It is of interest to know when $f^\prime$ is actually a factor of $f$.
Then Theorem~\ref{3} has an interesting variation as follows:

\begin{thm3'}
Let $f\colon X\rightarrow Y$ be a continuous surjection between compact Hausdorff spaces.
Then $f$ is open iff $f^\prime$ is open iff $f_*$ is open.
\end{thm3'}

\begin{proof}
In view of Theorem~A we need only prove the first ``iff''.
Suppose $f$ is open. Then $2^f$ is open by Theorem~\ref{3}. Set $\mathcal{Y}=\{\{y\}\,|\,y\in Y\}$. Then $2^{X,f}=(2^f)^{-1}[\mathcal{Y}]$. Thus, $f^\prime$ is open. Now conversely, if $f^\prime\colon2^{X,f}\rightarrow Y$ is open, then $f$ is obviously open. The proof is completed.
\end{proof}

\begin{thm}[{cf.~\cite[Lem.~2.3]{HZZ} for ``only if'' part}]\label{4}
Let $f\colon X\rightarrow Y$ be a continuous surjection between compact Hausdorff spaces. Then $f$ is semi-open iff $2^f\colon 2^X\rightarrow2^Y$ is semi-open.
\end{thm}

\begin{proof}
\item \textsl{Necessity\,(different with \cite{HZZ}):} Let $\mathcal {B}\subseteq 2^Y$ be a dense subset. By Lemma~\ref{2}a, we need only prove that $\mathcal {A}:=(2^f)^{-1}[\mathcal {B}]$ is dense in $2^X$.
For that, let $\langle U_1,\dotsc,U_n\rangle$ be a basic open set in $2^X$. It is enough to prove that $\mathcal {A}\cap\langle U_1,\dotsc,U_n\rangle\not=\emptyset$. Indeed, since $f$ is semi-open, hence
\begin{enumerate}
\item[] $V_i:=\mathrm{int}\,f[U_i]\not=\emptyset$, $i=1,\dotsc,n$.
\end{enumerate}
Let $V=V_1\cup\dotsm\cup V_n$, $U=U_1\cup\dotsm\cup U_n$.
Then $\langle V_1,\dotsc,V_n\rangle$ is an open non-empty subset of $2^Y$. So $\mathcal {B}\cap\langle V_1,\dotsc,V_n\rangle\not=\emptyset$. Let $B\in\mathcal {B}\cap\langle V_1,\dotsc,V_n\rangle$. As $V\subseteq f[U]$ and $B\in 2^Y$ is compact, it follows that we can select $A\in2^X$ with $A\subseteq U$ such that $f[A]=B$. Moreover, by $B\cap V_i\not=\emptyset$, we can take a point $x_i\in U_i$ with $f(x_i)\in B$ for $i=1,\dotsc,n$. Now set
\begin{enumerate}
\item[] $K=A\cup\{x_1,\dotsc,x_n\}$.
\end{enumerate}
Then $K\subseteq U$, $K\cap U_i\not=\emptyset$ for all $i=1, \dotsc, n$ and $f[K]=B$. Thus, $K\in\mathcal {A}$ and further $\mathcal {A}\cap\langle U_1,\dotsc,U_n\rangle\not=\emptyset$.

\item \textsl{Sufficiency:} Let $U\not=\emptyset$ be an open set in $X$. Since $\langle U\rangle=\{K\in2^X\,|\,K\subseteq U\}$ is open in $2^X$ and $2^f$ is semi-open, it follows that $2^f[\langle U\rangle]=\{f[K]\,|\,K\in\langle U\rangle\}$ has a non-empty interior. Therefore, there exists a non-empty basic open subset of $2^Y$, say $\langle V_1,\dotsc,V_n\rangle\subseteq 2^f[\langle U\rangle]$. Put
    \begin{enumerate}
    \item[] $B=V_1\cup\dotsm\cup V_n$.
    \end{enumerate}
    Then $B\not=\emptyset$ is open in $Y$ such that $B\subseteq f[U]$. For, as $B\in\langle V_1,\dotsc,V_n\rangle\subseteq2^f[\langle U\rangle]$, there exists
    $A\in\langle U\rangle$ such that $f[A]=B$; thus, $A\subseteq U$ implies $B\subseteq f[U]$. The proof is completed.
\end{proof}

Note here that we do not know whether or not Theorem~\ref{4} has a variation similar to Theorem~\ref{3}$^\prime$. That is, we do not know how to characterize the semi-openness of $f^\prime\colon2^{X,f}\rightarrow Y$.

\begin{se}\label{5}
A \textit{flow} is a triple $(T,X,\pi)$, simply denoted $\mathscr{X}$ or $T\curvearrowright X$ if no confusion, where $T$ is a topological group, called the phase group; $X$ is a topological space, called the phase space; and where $\pi\colon T\times X\xrightarrow{(t,x)\mapsto x}X$, the action, is a jointly continuous map, such that
$ex=x$ and $(st)x=s(tx)$ for all $x\in X$ and $s,t\in T$.
Here $e$ is the identity of $T$.

Every flow $\mathscr{X}$ can induce a so-called \textit{hyperflow} $(T,2^X,2^\pi)$\,(cf.~\cite{Ko} or \cite[Thm.~II.1.6]{Wo}), denoted $2^\mathscr{X}$, where the phase group is $T$, the phase space is $2^X$ and the action is the induced mapping $2^\pi\colon T\times 2^X\xrightarrow{(t,K)\mapsto tK}2^X$. Here $tK=\{tx\,|\,x\in K\}$ for $t\in T$ and $K\in 2^X$.

If $\overline{Tx}=X\ \forall x\in X$, then $\mathscr{X}$ is called a \textit{minimal flow}. Let $\mathscr{X}$ be a compact flow; that is, $\mathscr{X}$ is a flow with compact Hausdorff phase space $X$. If $x\in X$ such that $\overline{Tx}$ is a minimal subset of $\mathscr{X}$, then $x$ is refereed to as an \textit{almost periodic} (a.p.) \textit{point} for $\mathscr{X}$.
The induced affine flow $T\curvearrowright\mathfrak{M}(X)$ and hyperflow $2^\mathscr{X}$ of a compact minimal flow $\mathscr{X}$ are not minimal unless $X=\{pt\}$ a singleton. So, the dynamics on $2^\mathscr{X}$ is more richer than that on $\mathscr{X}$. See, e.g., \cite{BS, GW, LYY, AAN, LOYZ, JO}.

Let $\mathscr{X}, \mathscr{Y}$ be two compact flows. Then $\phi\colon\mathscr{X}\rightarrow\mathscr{Y}$ is called an \textit{extension} if $\phi\colon X\rightarrow Y$ is a continuous surjection such that $\phi(tx)=t\phi(x)$ for all $t\in T$ and $x\in X$.
If $\phi\colon \mathscr{X}\rightarrow \mathscr{Y}$ is an extension of compact flows, then $2^\phi\colon 2^\mathscr{X}\rightarrow2^\mathscr{Y}$ is an extension of compact hyperflows\,(see \cite[Thm.~II.1.8]{Wo}).
\end{se}

\begin{lem}[{cf.~\cite[Lem.~3.12.15]{B79} or \cite[Thm.~I.1.4]{Wo}}]\label{6}
If $\phi\colon\mathscr{X}\rightarrow\mathscr{Y}$ is an extension of compact flows with $\mathscr{X}$ having a dense set of a.p. points and $\mathscr{Y}$ minimal, then $\phi$ is semi-open.
\end{lem}

\begin{6A}
In Definition~\ref{5}, if $T$ is only a topological semigroup, then $\mathscr{X}$ or $T\curvearrowright X$ is refereed to as a \textit{semiflow}. In this case, $x\in X$ is called \textit{a.p.} for $\mathscr{X}$ iff $x\in \overline{Tx}$ and $\overline{Tx}$ is a minimal subset of $\mathscr{X}$. In fact, it is not known whether or not Lemma~\ref{6} is still true if $T$ is only a semigroup. See Theorem~\ref{6}C for a confirmative conditional case.
\end{6A}

\begin{6B}
Let $\phi\colon\mathscr{X}\rightarrow\mathscr{Y}$ be an extension of compact semiflows. We say that $\phi$ is \textit{highly proximal}\,(h.p.) iff, for all $y\in Y$, there is a net $t_n\in T$ with $t_n\phi^{-1}(y)\to\{pt\}$ in $2^X$ (cf.~\cite[p.104]{Wo}).
\end{6B}

\begin{6C}
Let $\phi\colon\mathscr{X}\rightarrow\mathscr{Y}$ be an extension of compact semiflows such that $t\phi^{-1}(y)=\phi^{-1}(ty)$ for all $t\in T, y\in Y$. If $\phi$ is h.p. and $\mathscr{X}$ has a dense set of a.p. points, then $\phi$ is semi-open.
\end{6C}

\begin{proof}
Firstly we claim that every nonempty open subset of $X$ contains a $\phi$-fiber. Indeed, let $U\subset X$ be open with $U\not=\emptyset$. Then there is an a.p. point $x_0\in U$. Since $\phi$ is h.p., there is a net $t_n\in T$ such that $t_n\phi^{-1}\phi(x_0)\to\{x^\prime\}$, with $x^\prime=\lim t_nx_0$, in $2^X$. As $x_0$ is a.p., it follows that there is a net $s_j\in T$ with $s_jx^\prime\to x_0$. Thus, there is a net $\tau_i\in T$ such that $\tau_i\phi^{-1}\phi(x_0)\to L\subseteq U$, and so that $\phi^{-1}\phi(\tau_ix_0)\subseteq U$ eventually.

Now to prove that $\phi$ is semi-open, let $U\subset X$ be open with $U\not=\emptyset$. By our claim above, there exists a point $y_0\in Y$ such that $\phi^{-1}(y_0)\subseteq U$. Since $\phi_\textrm{ad}|_Y\colon Y\rightarrow 2^X$, $y\mapsto\phi^{-1}(y)$ is upper semi-continuous and $U$ is open, there is an open neighborhood $V$ of $y_0$ in $Y$ such that $\phi^{-1}[V]\subseteq U$. Thus, $\phi[U]\supseteq V$ so that $\textrm{int}\,\phi[U]\not=\emptyset$. The proof is completed.
\end{proof}

Therefore, if $\phi\colon\mathscr{X}\rightarrow\mathscr{Y}$ is an h.p. extension of compact flows with $\mathscr{X}$ having a dense set of a.p. points, then $\phi$ is semi-open. A point of Theorem~\ref{6}C is that $\mathscr{Y}$ need be minimal here.

As was mentioned before, $2^\mathscr{X}$ need not have a dense set of a.p.~points and $2^\mathscr{Y}$ are generally not minimal, so Lemma~\ref{6} is not applicable straightforwardly for $2^\phi\colon 2^\mathscr{X}\rightarrow2^\mathscr{Y}$. However, using Theorem~\ref{4} we can obtain the following:

\begin{cor}\label{7}
Let $\phi\colon\mathscr{X}\rightarrow\mathscr{Y}$ be an extension of compact minimal flows. Then:
\begin{enumerate}[(1)]
\item $2^\phi\colon 2^{\mathscr{X}}\rightarrow2^{\mathscr{Y}}$ and $2^{2^\phi}\colon2^{2^{\mathscr{X}}}\rightarrow2^{2^{\mathscr{Y}}}$ both are semi-open extensions of compact hyperflows.
\item If $X$ and $Y$ both are compact metric spaces, then
\begin{enumerate}
\item[] $\phi_*\colon\mathfrak{M}(X)\rightarrow\mathfrak{M}(Y)$\quad and\quad $(\phi_*)_*\colon T\curvearrowright\mathfrak{M}(\mathfrak{M}(X))\rightarrow T\curvearrowright\mathfrak{M}(\mathfrak{M}(Y))$
    \end{enumerate}
    both are semi-open extensions of compact flows.
\end{enumerate}
\end{cor}

\begin{proof}
By Lemma~\ref{6}, Theorem~\ref{4} and Theorem~B$^\prime$.
\end{proof}

\begin{cor}\label{8}
Let $\phi\colon\mathscr{X}\rightarrow\mathscr{Y}$ be an extension of compact flows. Then
$\phi$ is open iff $2^\phi\colon2^\mathscr{X}\rightarrow2^\mathscr{Y}$ is open iff $2^{2^\phi}\colon2^{2^\mathscr{X}}\rightarrow2^{2^\mathscr{Y}}$ is open.
\end{cor}

\begin{proof}
By Theorem~3.
\end{proof}

\begin{se}[Irreducibility of maps]\label{9}
In what follows, let $\phi\colon X\rightarrow Y$ be a continuous surjection between compact Hausdorff spaces. We say that $\phi$ is \textit{irreducible} if the only member $A\in 2^X$ with $\phi[A]=Y$ is $X$ itself. This notion is closely related to ``highly proximal'' in extensions of minimal flows~\cite{Wo}.

\begin{lem9A}
$\phi$ is irreducible iff every non-empty open subset $U$ of $X$ contains a fiber $\phi^{-1}(y)$ for some point $y\in Y$.
\end{lem9A}

\begin{proof}
It is straightforward.
\end{proof}

\begin{thm9B}
$\phi$ is irreducible iff $2^\phi\colon2^X\rightarrow2^Y$ is irreducible.
\end{thm9B}

\begin{proof}
\item \textsl{Sufficiency:} Assume $2^\phi$ is irreducible. Let $A\in 2^X$ with $\phi[A]=Y$. Then $2^A\subseteq 2^X$ is a closed set such that $2^\phi[2^A]=2^Y$. Thus, $2^A=2^X$ and $A=X$. This shows that $\phi$ is irreducible.

\item \textsl{Necessity:} Suppose $\phi$ is irreducible. Let $\langle U_1,\dotsc,U_n\rangle$ be a basic open set in $2^X$. In view of Lemma~\ref{9}A, it is sufficient to prove that $\langle U_1,\dotsc,U_n\rangle$ includes a fiber of $2^\phi$. Indeed, as $\phi$ is irreducible, it follows by Lemma~\ref{9}A that there is a point $y_i\in\phi[U_i]$ with $\phi^{-1}(y_i)\subseteq U_i$ for all $i=1,\dotsc,n$. Let $B=\{y_1,\dotsc,y_n\}\in 2^Y$. Then
    $(2^\phi)^{-1}(B)\subseteq\langle U_1,\dotsc,U_n\rangle$.
    Thus, $2^\phi$ is irreducible. The proof is completed.
\end{proof}
\end{se}

\begin{se}\label{10}
Let $f\colon X\rightarrow Y$ be a semi-open continuous surjection of compact Hausdorff spaces, where $X$ is not metrizable. In view of Theorem~B$^{\prime\prime}$ we naturally wonder whether or not $f_*\colon\mathfrak{M}(X)\rightarrow\mathfrak{M}(Y)$ is semi-open.
The following lemma seems to be helpful for the this question, which is of interest independently.

\begin{10A}
Let $f\colon X\rightarrow Y$ be a continuous surjection between compact Hausdorff spaces, where $X$ is not necessarily metrizable. Let $\psi\colon X\rightarrow\mathbb{R}$ be a continuous function and define functions
\begin{enumerate}
\item[] $\psi^*\colon Y\xrightarrow{y\mapsto\psi^*(y)=\sup_{x\in f^{-1}(y)} \psi(x)}\mathbb{R}$\quad and\quad
$\psi_*\colon Y\xrightarrow{y\mapsto\psi_*(y)=\inf_{x\in f^{-1}(y)} \psi(x)}\mathbb{R}$.
\end{enumerate}
Then there is a residual set $Y_c(\psi)\subseteq Y$ such that $\psi^*$ and $\psi_*$ are continuous at each point of $Y_c(\psi)$.
\end{10A}

\begin{proof}
Let $\rho$ be a continuous pseudo-metric on $X$ and let $\tilde{\rho}$ be the naturally induced one on $2^X$. Since $f^{-1}\colon Y\rightarrow2^X$, defined by $y\mapsto f^{-1}y$, is upper semi-continuous, hence $f^{-1}\colon Y\rightarrow(2^X,\tilde{\rho})$ is also upper semi-continuous. Thus, there exists a residual set $Y_\rho\subseteq Y$ such that $f^{-1}$ is continuous at every point of $Y_\rho$.
Now, for every $\varepsilon>0$, there exists a continuous pseudo-metric $\rho$ on $X$ and a positive $r>0$ such that if $x,x^\prime\in X$ with $\rho(x,x^\prime)<r$, then $|\psi(x)-\psi(x^\prime)|<\varepsilon/3$. Then there exists a residual set $Y_\varepsilon=Y_\rho\subseteq Y$ such that for every $y\in Y_\varepsilon$, we have that $|\psi^*(y)-\psi^*(y^\prime)|+|\psi_*(y)-\psi_*(y^\prime)|<\varepsilon$ as $y^\prime\in Y$ close sufficiently to $y$.
Let $Y_c=\bigcap_{n=1}^\infty Y_{1/n}$. Clearly, $Y_c$ is a residual subset of $Y$ as desired. The proof is completed.
\end{proof}

\begin{10B}[Densely open mappings]
Let $f\colon X\rightarrow Y$ be a continuous surjection between compact Hausdorff spaces. We say that $f$ is \textit{densely open} if there exists a dense set $Y_o\subseteq Y$ such that $f^{-1}\colon Y\rightarrow 2^X$ is continuous at each point of $Y_o$. For example, if $X$ is a separable metric space, then $f$ is always densely open.
We notice here that the ``densely open'' is different with ``almost open'' considered in \cite[Def.~1.5]{HZZ} and \cite{AC}.
\end{10B}

\begin{10C}
Let $f\colon X\rightarrow Y$ be a continuous surjection between compact Hausdorff spaces. Then:
\begin{enumerate}[(1)]
\item If $f$ is semi-open densely open, then $f_*$ is semi-open densely open.
\item If $f$ is densely open, then $f_*$ is densely open.
\end{enumerate}
\end{10C}

\begin{proof}
Based on Def.~\ref{10}B, let $Y_o$ be the dense set of points of $Y$ at which the set-valued map $f^{-1}\colon Y\rightarrow2^X$ is continuous. Let $\delta_Y=\{\delta_y\,|\,y\in Y\}$.

\item (1):  As $f$ is semi-open, it follows from Lemma~\ref{2} that $X_o:=f^{-1}[Y_o]$ is dense in $X$.
In view of (2), we need only prove that $f_*$ is semi-open.
For that let $\mathcal{U}\subset\mathfrak{M}(X)$ be a closed set with non-empty interior. We need only show that $\mathrm{int}\,f_*[\mathcal{U}]\not=\emptyset$. Suppose to the contrary that $\mathrm{int}\,f_*[\mathcal{U}]=\emptyset$. We now fix a measure
    \begin{enumerate}
    \item[] $\mu_0=\sum_{i=1}^mc_i\delta_{x_i}\in\mathrm{int}\,\mathcal{U}$ with $x_i\in X_o$, $0<c_i\le1$, $\sum_ic_i=1$.
    \end{enumerate}
    Set
    \begin{enumerate}
    \item[] $\nu_0=f_*(\mu_0)=\sum_{i=1}^mc_i\delta_{y_i}$ where $y_i=f(x_i)\in Y_o$.
    \end{enumerate}
    Let
    \begin{enumerate}
    \item[] $\nu_j=\sum_{i=1}^{m_j} c_{j,i}\delta_{y_{j,i}}\in\mathrm{co}(\delta_Y)\setminus f_*[\mathcal{U}]$ be a net with $\nu_j\to \nu_0$.
    \end{enumerate}
    So $\nu_j\to\nu_0$.
Each $Q_j$ is a closed convex subset of $\mathfrak{M}(X)$ and, with no loss of generality, we may assume that $Q=\lim Q_j$ exists in $2^{\mathfrak{M}(X)}$. Then $Q$ is a compact convex subset of $\mathfrak{M}(X)$ with $f_*[Q]=\{\nu_0\}$.

If $\mu_0\in Q$, then $Q_j\cap\mathcal{U}\not=\emptyset$ eventually so that $\nu_j\in f_*[\mathcal{U}]$, contradicting our choice of $\nu_j$. Thus, $\mu_0\notin Q$ and by the Separation Theorem there exists a $\psi\in C(X)$ and $\epsilon>0$ such that
    \begin{enumerate}
    \item[] $\mu_0(\psi)\ge q(\psi)+\epsilon\ \forall q\in Q$.
    \end{enumerate}
Define the associated function
\begin{enumerate}
\item[] $\psi^*\colon Y\rightarrow\mathbb{R}$, by $y\mapsto\psi(y)=\sup\{\psi(x)\,|\,x\in f^{-1}y\}$, such that $\psi^*\circ f\ge\psi$.
\end{enumerate}
We can choose points $x_{j,i}\in X$ with $f(x_{j,i})=y_{j,i}$ and $\psi^*(y_{j,i})=\psi(x_{j,i})$. Now form the measures
$\mu_j=\sum\limits_{i=1}^{m_j}c_{j,i}\delta_{x_{j,i}}$
and assume, with no loss of generality, that
$\mu=\lim_j\mu_j$
exists in $\mathfrak{M}(X)$. Since $\mu_j\in Q_j$ for each $j$, we have $\mu\in Q$. Thus,
$\mu_0(\psi)\ge \mu(\psi)+\epsilon$.
By our construction $\nu_j(\psi^*)=\mu_j(\psi)$ for every $j$. By assumption we have that $\textrm{supp}(\nu_0)=\{y_1,\dotsc,y_m\}$ is a subset of $Y_o$ and therefore, each $y_i$ is a continuity point of $\psi^*$. Thus, $\lim\nu_j(\psi^*)=\nu_0(\psi^*)$. It then follows that
    $$
    \mu(\psi)=\lim_j\mu_j(\psi)=\lim_j\nu_j(\psi^*)=\nu_0(\psi^*)=\sum_{i=1}^mc_i\psi^*(y_i)\ge\sum_{i=1}^mc_i\psi(x_i)=\mu_0(\psi).
    $$
    This contradicts the choice of $\psi$. Therefore $f_*$ is semi-open.

\item (2): We will now prove that $f_*$ is densely open. Suppose to the contrary that $f_*$ is not densely open. Then there exists an open set $\mathcal {V}\not=\emptyset$ in $\mathfrak{M}(Y)$ such that $f_*^{-1}\colon\mathfrak{M}(Y)\rightarrow2^{\mathfrak{M}(X)}$ is not continuous at every point of $\mathcal {V}$. Since $\mathrm{co}(\delta_{Y_o})$ is dense in $\mathfrak{M}(Y)$, there exists a measure of the form
    \begin{enumerate}
    \item[] $\nu_o=\sum_{i=1}^mc_i\delta_{y_i}$, where $y_i\in Y_o$ and $0\le c_i\le 1$ with $\sum c_i=1$,
    \end{enumerate}
    such that $\nu_0\in\mathcal {V}$. This implies that there is a measure
    \begin{enumerate}
    \item[] $\mu_o\in f_*^{-1}(\nu_o)$ such that there is $\mathcal {U}\in\mathfrak{N}_{\mu_o}(\mathfrak{M}(X))$ with $f_*[\mathcal {U}]\notin\mathfrak{N}_{\nu_o}(\mathfrak{M}(Y))$.
    \end{enumerate}
Then we can select a net $\nu_j=\sum_{i=1}^{m_j} c_{j,i}\delta_{y_{j,i}}\in\mathrm{co}(\delta_{Y_o})\setminus f_*[\mathcal{U}]$ such that $\nu_j\to\nu_o$. Next by an argument as in (1) we can reach a contradiction.
The proof is completed.
\end{proof}
\end{se}

\section*{Acknowledgements}
This work was supported by National Natural Science Foundation of China (Grant No. 12271245) and PAPD of Jiangsu Higher Education Institutions.



\begin{thebibliography}{10}
\bibitem{AAN}
\newblock {E.~Akin, J.~Auslander and A.~Nagar},
\newblock {\it Dynamics of induced systems},
\newblock {Ergod. Th. $\&$ Dynam. Sys., 37 (2017), 2034-2059}.

\bibitem{AC}
\newblock {A.~Andrade and J.~Camargo},
\newblock {\it A note on semi-open and almost open induced maps},
\newblock {Topology Appl., 302 (2021), No. 107823, 6 pp}.

\bibitem{BS}
\newblock {W. Bauer and K. Sigmund},
\newblock {\it Topological dynamics of transformations induced on space of probability measures},
\newblock {Monatsh. Math., 79 (1975), 81--92}.


\bibitem{B79}
\newblock I.\,U.~Bron\v ste\v in,
\newblock {\it Extensions of Minimal Transformation Groups},
\newblock {Sijthoff $\&$ Noordhoff, 1975/79}.

\bibitem{C}
\newblock {J.\,J.~Charatonik},
\newblock {\it Recent research in hyperspace theory},
\newblock {Extracta Math., 18 (2003), 235-262}.

\bibitem{CIM}
\newblock {J.\,J.~Charatonik, A.~Illanes and S.~Mac\'{\i}as},
\newblock {\it Induced mappings on hyperspace $C_n(X)$ of a continuum $X$},
\newblock {Houston J. Math., 28 (2002), 781-805}.

\bibitem{DE}
\newblock {S.\,Z. Ditor and L.\,Q. Eifler},
\newblock {\it Some open mapping theorems for measures},
\newblock {Trans. Amer. Math. Soc., 164 (1972), 287-293}.

\bibitem{G07}
\newblock {E. Glasner},
\newblock {\it The structure of tame minimal dynamical systems},
\newblock {Ergod. Th. $\&$ Dynam. Sys., 27 (2007), 1819-1837}.

\bibitem{G19}
\newblock {E. Glasner},
\newblock {\it The structure of tame minimal dynamical systems for general groups},
\newblock {Invent. Math., 211 (2018), 213--244}.

\bibitem{GW}
\newblock {E. Glasner and B.~Weiss},
\newblock {\it Quasi-factors of zero-entropy systems},
\newblock {J. Amer. Math. Soc., 8 (1995), 655-686}.

\bibitem{H97}
\newblock {H. Hosokawa},
\newblock {\it Induced mappings on hyperspaces},
\newblock {Tsukuba J. Math., 21 (1997), 239--250}.

\bibitem{HZZ}
\newblock {X.~Huang, F.~Zeng and G.~Zhang},
\newblock {\it Semi-openness and almost-openness of induced mappings},
\newblock {Appl. Math. J. Chinese Univ. Ser. B, 20 (2005), 21--26}.

\bibitem{IN}
\newblock {A. Illanes and S.\,B. Nadler, Jr.},
\newblock {\it Hyperspaces, Fundamentals and Recent Advances},
\newblock {New York, Marcel Dekker Inc, 1999}.

\bibitem{JO}
\newblock {D.~Jeli\'{c} and P.~Oprocha},
\newblock {\it On recurrence and entropy in the hyperspace of continua in dimension one},
\newblock {Fund. Math., 263 (2023), 23--50}.

\bibitem{K42}
\newblock {J.\,L. Kelley},
\newblock {\it Hyperspaces of a continuum},
\newblock {Trans. Amer. Math. Soc., 52 (1942), 22--36}.

\bibitem{Ko} 
\newblock {S.\,C. Koo},
\newblock {\it Recursive properties of transformations groups in hyperspaces},
\newblock {Math. Systems Theory, 9 (1975), 75--82}.

\bibitem{LOYZ}
\newblock {J.~Li, P.~Oprocha, X.~Ye and R.~Zhang},
\newblock {\it When are all closed subsets recurrent?},
\newblock {Ergod. Th. $\&$ Dynam. Sys., 37 (2017), 2223-2254}.

\bibitem{LYY}
\newblock {J.~Li, K.~Yan and X.~Ye},
\newblock {\it Recurrence properties and disjointness on the induced spaces},
\newblock {Discrete Contin. Dyn. Syst., 35 (2015), 1059-1073}.


\bibitem{Mi}
\newblock {E. Michael},
\newblock {\it Topologies on spaces of subsets},
\newblock {Trans. Amer. Math. Soc., 71 (1951), 152--182}.

\bibitem{Wo}
\newblock J.\,C.\,S.\,P. van der~Woude,
\newblock \textit{Topological Dynamix},
\newblock CWI Tracts No. 22, Centre for Mathematics and Computer Science (CWI), Amsterdan, 1986.

\bibitem{V77}
  \newblock {W.\,A.~Veech},
  \newblock {\it Topological dynamics},
  \newblock {Bull. Amer. Math. Soc., 83 (1977), 775--830}.
\end{thebibliography}
\end{document}